\documentclass[preprint,11pt]{elsarticle}
\usepackage{amsfonts}
\usepackage{}
\usepackage{mathrsfs}
\usepackage{amsmath}
\usepackage{amssymb}
\usepackage{latexsym}
\usepackage{graphicx}
\usepackage{bm}
\usepackage{nicefrac}
\usepackage[usenames]{color}

\input amssym.def
\newsymbol\rtimes 226F
\newfont{\nset}{msbm10}

\newtheorem{theo}{Theorem}[section]
\newtheorem{theorem}[theo]{Theorem}

\newtheorem{lemma}[theo]{Lemma}
\newtheorem{definition}[theo]{Definition}

\journal{Theoretical Computer Science}

\begin{document}

\begin{frontmatter}

\title{THE $K$-POWER DOMINATION NUMBER IN SOME SELF-SIMILAR GRAPHS}

\author[lable2]{Yulun Xu}

\author[lable1,lable3]{Qi Bao}

\author[lable1,lable3]{Zhongzhi Zhang}

\ead{zhangzz@fudan.edu.cn}

\address[lable1]{Shanghai Key Laboratory of Intelligent Information
	Processing, Fudan
	University, Shanghai 200433, China}
\address[lable2]{School of Mathematical Sciences, Fudan University, Shanghai 200433, China}
\address[lable3]{School of Computer Science, Fudan
	University, Shanghai 200433, China}
\begin{abstract}

The $k$-power domination problem is a problem in graph theory, which has applications in many areas. However, it is hard to calculate the exact $k$-power domination number since determining k-power domination number of a generic graph is a NP-complete problem. We determine the exact $k$-power domination number in two graphs which have the same number of vertices and edges: pseudofractal scale-free web and Sierpi\'nski gasket. The $k$-power domination number becomes 1 for $k\ge2$ in the Sierpi\'nski gasket, while the $k$-power domination number increases at an exponential rate with regard to the number of vertices in the pseudofractal scale-free web. The scale-free property may account for the difference in the behavior of two graphs.

\end{abstract}

\end{frontmatter}

\section{Introduction}

Let $\mathcal{V}$ be the vertex set of a graph $G$. Dominating Set (DS) is intensively studied in the graph theory. The basic problem is to find a subset $\mathcal{D}$ of $\mathcal{V}$ so that each vertex $v$ in $\mathcal{V}$ is in $\mathcal{D}$ or $v$ is a neighbor of a vertex in $\mathcal{D}$. In this paper, we focus on a variant of DS problem: $k-$power Dominating Set (k-PDS) problem. This problem is motivated by the need to decide the minimum number of \textit{phase measurement units}(PMU) necessary to monitor an electric power network ~\cite{haynes_domination_2002}. k-PDS problem is different from DS problem by having additional propagation originating from Kirschoff laws.\\

The open neighborhood of a vertex $v$, denoted as $N(v)$ is a set of vertices incident to $v$. The closed neighborhood of a vertex $v$ is $N[v]=N(v)\cup \{v\}$. The open neighborhood (respectively closed neighborhood) of $\mathcal{D}$, denoted as $N(\mathcal{D})$(respectively $N[\mathcal{D}]$) is the union of open neighborhood(respectively closed neighborhood) of elements in $\mathcal{D}$.\\

Denote $M(\mathcal{D})$ as the set of vertices that are \textit{k-power dominated} by $\mathcal{D}$, which is obtained algorithmically as follows, see ~\cite{ dorbec_power_2008, dorfling_note_2006, guo_improved_2008, wang_power_2005}:\\

$\begin{array}{l}{\text{1. Initialize } M(\mathcal{D})=N[\mathcal{D}] \text{;}} \\ {\text{2. (propagation) if a vertex } v \text{ is in } M(\mathcal{D}), \text{ and at most } k \text{ of its neighbors}} \\ {\text{ are not in } M(\mathcal{D}), \text{ then all the neighbors of } v \text{ are inserted into } M(\mathcal{D}) \text{. }}\\ {\text{3. If no new vertex can be located in the step 2 for inclusion, stop. }} \\{\text{Otherwise, go to step 2. }}\end{array}$\\

In other words, in the step 1, $M(\mathcal{D})$ initially contains closed neighborhoods of elements of $\mathcal{D}$. Then in the step 2, $M(\mathcal{D})$ is iteratively enlarged by adding all vertices $w\in \mathcal{V}$ which are incident to a vertex $v$ that has at most k neighbors not in $M(\mathcal{D})$. The step 2 is continued until there is no vertex can be added to $M(\mathcal{D})$. Then we have the set, $M(\mathcal{D})$, \textit{k-power dominated} by $\mathcal{D}$. $\mathcal{D}$ is called a k-power domination set of $G$ if $M(\mathcal{D})=\mathcal{V}$. The k-power domination number of $G$ is the minimum cardinality among k-power domination sets of $G$.\\

The formal definition of k-power dominating set is as follows:\\

\begin{definition}
	Let k be a nonnegative integer. Let $\mathcal{D}$ be a subset of vertex set $\mathcal{V}$ of $G$. Then the sets $(P_{G,k}^i(\mathcal{D}))$ that are k-power dominated by $S$ at step i, for $i\in\{0,1,2,...\}$, are defined as follows:\\
	
	$\begin{array}{l}{\mathcal{P}_{G, k}^{0}(\mathcal{D})=N_{G}[\mathcal{D}],(\text { domination })} \\ {\mathcal{P}_{G, k}^{i+1}(S)=\cup\left\{N_{G}[v]: v \in \mathcal{P}_{G, k}^{i}(S),\left|N_{G}[v] \backslash \mathcal{P}_{G, k}^{i}(S)\right| \leq k\right\}(\text { propagation })}\end{array}$
\end{definition}

Note that $P_{G,k}^i\subset P_{G,k}^{i+1}$. Suppose vertex set $\mathcal{V}$ of $G$ only has finite elements, there exists an integer $i_0$ such that $P_{G,k}^{j}=P_{G,k}^{i_0}. \text{ for } j\ge i_0$. Set $P_{G,k}^{\infty}=P_{G,k}^{i_0}$.\\

The $k$-power Dominating Set problem has received considerable attention. A quantity of previous works focused on $k$-power domination number on various media. It is pointed out that the problem is difficult, since it is a NP-complete problem even when it is restricted to bipartite graphs or chordal graphs, according to ~\cite{haynes_domination_2002}. Efficient algorithms for finding minimum k-power dominating sets have been developed for hypertrees ~\cite{chang_$$varveck$$_2015} and Circular-Arc Graphs ~\cite{liao_power_2013}. Exact k-power domination numbers are determined in some products of paths ~\cite{dorbec_power_2008} and some interconnection networks ~\cite{ rajan_2-power_2015}. Bounds of k-power domination number are obtained in hypertrees ~\cite{chang_$$varveck$$_2015}, Cartesian products of graphs and Petersen graphs ~\cite{barrera_power_2011} and some products of paths ~\cite{dorbec_power_2008}. In addition, scale-free phenomenon, which means that degree of vertices follows a power-law distribution: $P(k) \sim k^{-\gamma}$, is found in many real world problems, such as earthquake ~\cite{abe_scale-free_2004-1} and stock price changes ~\cite{kim_weighted_2002}. However, so far there is no work relating scale-free networks with k-PDS.\\

The omnipresence of power-law phenomenon in real world problems makes it fascinating to find the relation between the power-domination number and scale-free behavior, because this may lead to some applications of the $k$-power domination problem. Since it is difficult to determine the exact k-power domination number of a generic graph, it is intriguing to determine k-power domination number of some special graphs.\\

In this paper, we focus on the k-power domination number in a scale-free graph, called pseudofractal scale-free web ~\cite{dorogovtsev_pseudofractal_2002} ~\cite{zhang_exact_2009}, and the Sierpi\'nski gasket with the same number of vertices and edges. We choose these two graphs to study the dependence of k-power domination property on the scale-free behavior. We determine the exact number of k-power domination number in the pseudofracal scale-free web and Sierpi\'nski gasket. In the pseudofracal scale-free web, the $k$-power domination number increases as an exponential function of the vertices number, while in Sierpi\'nski gasket, the k-power dominating number is 1 for $k\ge2$. The difference between k-power domination numbers of these graphs lies in their distinct architecture: pseudofractal scale-free web is scale-free while Sierpi\'nski gasket is not.\\
 
\section{$k$-power domination in pseudofractal scale-free web}

In this section, we determine the $k$-power domination number in the pseudofractal scale-free web for every positive integer $k$.

\subsection{Network construction and properties}

We construct the pseudofractal scale-free web by induction. Denote $\mathcal{G}_g, g \ge 1$, as the g-generation network. Define $\mathcal{G}_1$ as a triangle with three vertices and three edges. Suppose $\mathcal{G}_g$ is constructed, then $\mathcal{G}_{g+1}$ is constructed by adding a vertex linked to both end vertices of every existent edge. Fig.~\ref{network} shows the first three generations of the scale-free network.

The construction of the network shows the prominent properties observed in various real-life systems. First, the pseudofractal scale-free web is scale-free, because the degree distribution of its vertices obeys a power law form $P(n)\sim n^{1+\frac{ln3}{ln2}}$. Besides, it reveals the small-world effect, since its average distance goes up logarithmically with the number of vertices and when the average clustering converges, it converges to a constant.

Self-similarity is another fascinating property of the scale-free web. This is a pervasive property in the realistic network. We define the initial three vertices as hub vertices, denoted as $A_{g}$, $B_{g}$, and $C_{g}$, respectively. We show this property of the scale-free web by giving another construction method. Given the gth generation network $\mathcal{G}_g$, $\mathcal{G}_{g+1}$ can be obtained by merging three copies of $\mathcal{G}_g$ at their hub vertices.

\begin{figure}
\begin{center}
\includegraphics[width=0.8\linewidth]{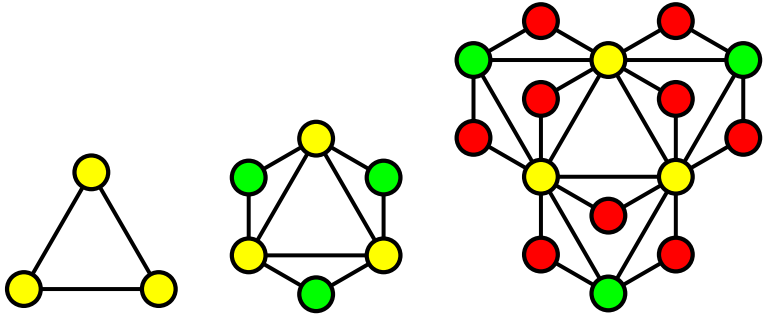}
\end{center}
\caption[kurzform]{The first three generations of the
scale-free network.} \label{network}
\end{figure}

Let $\mathcal{G}_{g}^{\theta}$, $\theta=1,2,3$, be three copies of $\mathcal{G}_{g}$, the hub vertices of which are represented by $A_{g}^{\theta}$, $B_{g}^{\theta}$,
and $C_{g}^{\theta}$, respectively. Then, $\mathcal{G}_{g+1}$ can be obtained by joining $\mathcal{G}_{g}^{\theta}$, with $A_{g}^{1}$
(resp. $C_{g}^{1}$, $A_{g}^{2}$) and $B_{g}^{3}$ (resp. $B_{g}^{2}$,
$C_{g}^{3}$) being identified as the hub vertex $A_{g+1}$ (resp.
$B_{g+1}$, $C_{g+1}$) in $\mathcal{G}_{g+1}$.

\subsection{$K$-power domination number of pseudofractal scale-free web}

First we determine the $k$-power domination number in an easy case. To do that we define a condition about $k$-power dominating sets of pseudofractal scale-free web, because when we use the mathematical induction to determine the $k$-power domination number, a subset of $V(\mathcal{G}_g^{\theta})$ that corresponds to a k-power domination set of $\mathcal{G}_g$ may not still k-power dominates $\mathcal{G}_g^{\theta}$ in $\mathcal{G}_{g+1}$, which is a result of the fact that the degrees of two of hub vertices of $\mathcal{G}_g^{\theta}$ in $\mathcal{G}_{g+1}$ will be bigger than those of $\mathcal{G}_g$, so we set one of these hub vertices as a $k$-power domination set in the definition of condition 1 and we don't use the other hub vertex to monitor its neighborhood in the propagation steps. Thus the  $k$-power dominating set of $\mathcal{G}_g$ in the condition 1 will still k-power dominates $\mathcal{G}_g^{\theta}$ in $\mathcal{G}_{g+1}$.

\begin{definition}
A $k$-power dominating set $S$ of $\mathcal{G}_g$ satisfies the condition 1 if there exists a sequence of subsets $\{D_i\}_{i=1}^{\infty}$ of V($\mathcal{G}_g$) satisfying the following conditions:
\par(1)$P(D_1)=N_{\mathcal{G}_g}[D_1]. $ $P(D_i)=D_i\cup \{N_{\mathcal{G}_g}[v]: v\neq C_g, v\in D_i, |N_{\mathcal{G}}[v]\backslash D_i|\leq k\}$ for $i\ge 2$. We have $D_i\subset D_{i+1}\subset P(D_i)$,
\par(2)$D_1=S=\{A_g\}$,
\par(3)There exists i so that $D_i=V(\mathcal{G}_g)$.
\end{definition}

Remark: a k-power domination set S is also said to satisfy the condition 1 if $A_g$ and $C_g$ in the definition above are substituted by any two different vertices among $A_g, B_g$ and $C_g$. Fig.~\ref{condition 1} shows a 3-power domination set in $\mathcal{G}_3$ that satisfies the condition 1\\

\begin{figure}
	\begin{center}
		\includegraphics[width=0.8\linewidth]{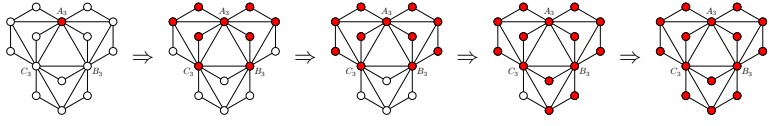}
	\end{center}
	\caption[kurzform]{The illustration of a 3-power domination set in $\mathcal{G}_3$ that satisfies the condition 1.} \label{condition 1}
\end{figure}

\begin{lemma}\label{leSGDom01}
 If integers $k$ and $g$ satisfy $k\ge 2^{g-1}-1$, $g\ge 1$, then the k-power domination number of $\mathcal{G}_g$  is 1. To be more specific, $P_{\mathcal{G}_g,k}^{\infty} (\{A_n\})=V(\mathcal{G}_g)$
\end{lemma}
\begin{proof}It  suffices to prove that there exists a k-power dominating set $S$ of $\mathcal{G}_g$ satisfying the condition 1 for $k\ge 2^{g-1}-1$ by induction on $g$. When $g=1$, we can verify the assertion by hand. Thus the basis step holds immediately. Then suppose the assertion holds for $g=t$, $t\ge 1$, $k\ge 2^t-1$. Because $A_{t+1}$ coincides with $A_{t}^1$ and $B_{t}^3$ and  $k\ge 2^{t}-1\ge 2^{t-1}-1$, using the assertion for $g=t$ we have that $\{A_{t+1}\}$ satisfies the condition 1 in $\mathcal{G}_{t}^1$. Because $A_{t}^1$ and $C_{t}^1$ are the only vertices in $\mathcal{G}_{t}^1$ that have bigger degrees than the corresponding vertices of $\mathcal{G}_{t}$ and the neighbors of $A_{t+1}^1$ are monitored by $A_{t+1}^1$, according to the definition of the condition 1, we have that $V(\mathcal{G}_{t}^1)\subset  P_{\mathcal{G}_{t+1},k}^{\infty}(\{A_{t+1}\})$. For the same reason, we have that $V(\mathcal{G}_{t}^3)\subset  P_{{\mathcal{G}_{t+1}},k}^{\infty}(\{A_{t+1}\})$
Thus we have that
$V(\mathcal{G}_{t}^1)\cup V(\mathcal{G}_{t}^3) \subset P_{\mathcal{G}_{t+1},k}^\infty (\{A_{t+1}\})$. Since $(\{B_{t+1}, C_{t+1}\}\cup V(\mathcal{G}_{t}^1)\cup V(\mathcal{G}_{t}^3))\subset P_{\mathcal{G}_{t+1}, k}^{\infty}(\{A_{t+1}\})$ and $|N(B_{t+1})|=2^{t+1}$, we have that $|N(B_{t+1})
\setminus P_{\mathcal{G}_{t+1}, k}^{\infty}(\{A_{t+1}\})|\leq 2^{t}-1$. Because $k\ge 2^{t}-1$ by hypothesis, $N(B_{t+1})\subset P_{\mathcal{G}_{t}, k}^{\infty}({A_{t+1}})$. Then like what we have discussed before about $A_{t+1}$ in $\mathcal{G}_{t}^1$, we have $V(\mathcal{G}_{t}^2)=P_{\mathcal{G}_{t}^2,k}^{\infty}({B_{t}^2})\subset P_{\mathcal{G}_{t+1},k}^{\infty}({A_{t+1}}).$ Then we have proved that $P_{\mathcal{G}_{t+1}, k}^\infty (\{ A_{t+1} \} )=V(\mathcal{G}_{t+1})$ holds for $g=t+1$.\\
This concludes the proof of the lemma.
\end{proof}~\\

Next we determine the $k$-power domination number of the pseudofractal scale-free web in a general case. To do this, we first define a new sequence of graphs. Then we use the vertex cover number of this sequence of graphs to determine the $k$-power domination number of the pseudofractal scale-free web.

\begin{definition}\label{$I_g$}
Let $I_1$ denote the graph with two vertices linked by an edge. Let $I_g$ denote the $(g-1)$-generation network $\mathcal{G}_{g-1}$ for $g\ge 2$. Fig.~\ref{Ig} illustrates $I_1$, $I_2$ and $I_3$. 
\end{definition}

\begin{figure}
	\begin{center}
		\includegraphics[width=0.8\linewidth]{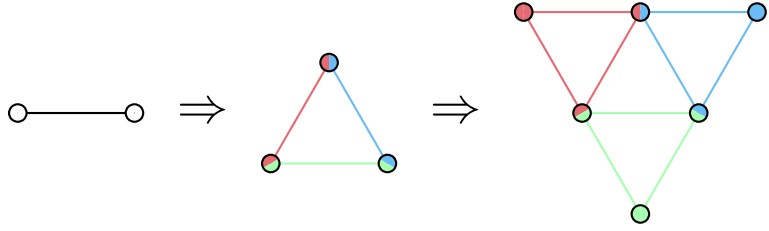}
	\end{center}
	\caption[kurzform]{The first three generations of $I_g$.} \label{Ig}
\end{figure}

For any $g\ge n$, $\mathcal{G}_g $ can be obtained by merging $3^{g-n}$ replicas of $\mathcal{G}_n$, $\mathcal{G}_n^\theta, \theta =1,2,3,...,3^{g-n}$ at their hub vertices, see Fig.~\ref{FG-1}\\

Next We prove that the vertex cover number of $I_g$ gives the $k$-power domination number of pseudofractal scale-free web.

\begin{lemma}\label{leSGDom02}
Let $\phi_g$  be the vertex cover number of $I_g$. Let $\gamma_{P,k}(\mathcal{G}_g)$ be the k-power domination number of $\mathcal{G}_g$. If integers g, k and n satisfy $g\ge n$ and $2^n-2\ge k\ge 2^{n-1}-1$, we have $\gamma_{P,k}(\mathcal{G}_g)=\phi_{g-n+1}$.
\end{lemma}
\begin{proof} $\mathcal{G}_g$ can be constructed by making $3^{g-n}$ replicas of $\mathcal{G}_g$ and merging them at their hub vertices.
\begin{figure}
	\begin{center}
		\includegraphics[width=0.8\linewidth]{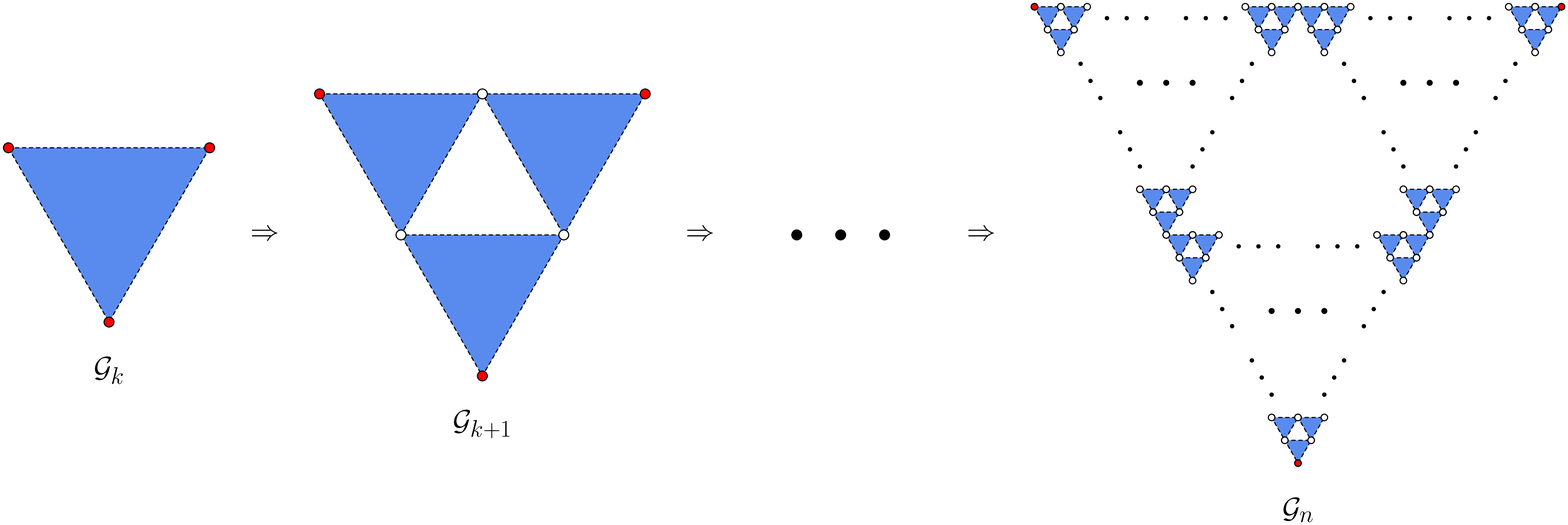}
	\end{center}
	\caption[kurzform]{The illustration of the constuction of $\mathcal{G}_g$ with replicas of $\mathcal{G}_k$, $k\le g$.} \label{FG-1}
\end{figure}

For every $\theta _0$, there are two vertices of $\mathcal{G}_n^{\theta_0}$ whose neighborhoods include some vertices that are not in $\mathcal{G}_n^{\theta_0}$ and the degrees of these two vertices are at least $2^{n+1}$. Then every vertex in $[\bigcup_{\theta=1}^{\theta_0-1}V(\mathcal{G}_n^\theta)]\cup [\bigcup_{\theta=\theta_0+1}^{3^{g-n}}V(\mathcal{G}_n^\theta)]$ has zero or at least  $2^n-1$ adjacent vertices that are in $V(\mathcal{G}_{n+1})\setminus [\bigcup_{\theta=1}^{\theta_0-1}V(\mathcal{G}_n^\theta)]\cup[\bigcup_{\theta=\theta_0+1}^{3^{g-n}}V(\mathcal{G}_N^\theta)])$. Since $2^n-2\le k\le 2^{n-1}-1$, if $S\cap \mathcal{G}_n^\theta=\emptyset$, then
\begin{equation}
P_{\mathcal{G}_g,k}^\infty(S)\cap\{V(\mathcal{G}_g)\setminus[\bigcup_{\theta=1}^{\theta_0-1}V(\mathcal{G}_n^\theta)]\cup [\bigcup_{\theta=\theta_0+1}^{3^{g-n}}V(\mathcal{G}_n^\theta)]\}=\emptyset.
\end{equation}
So if S is a k-power domination set of $\mathcal{G}_g$, then $S\cap V(\mathcal{G}_n^{\theta_0})\neq \emptyset$. According to the Lemma~\ref{leSGDom01} , $\gamma_{P, k}^{\infty}(\mathcal{G}_n^{\theta_0})=1$. So there is only one vertex in $S\cap V(\mathcal{G}_n^{\theta_0})$. If the vertex $v_n^{\theta_0}$ in $S\cap V(\mathcal{G}_n^{\theta_0})$ is not one of the hub vertices that is in $V(\mathcal{G}_n^{\theta_0})\cap V(\mathcal{G}_n^{\theta_1})$ for some $\theta_1$. We can substitute $v_n^{\theta_0}$ with one of such hub vertices. The new set is also a k-power domination set of $\mathcal{G}_g$. Through the above procedure, we can obtain a dominating set S', each of whose vertices is in $V(\mathcal{G}_n^{\theta_1})\cap V(\mathcal{G}_n^{\theta_2})$ for some $\theta_1$ and $\theta_2$. Since $\{V(\mathcal{G}_n^{\theta_1})\cap V(\mathcal{G}_n^{\theta_2}): 1\le \theta_1 < \theta_2\le 3^{g-n}\}$ and the edges between these vertices constitute the graph $I_{g-n+1}$, so we can see that $\gamma_{P,k}(\mathcal{G}_g)=\phi_{g-n+1}$
\end{proof}~\\

In order to determine $\phi_g$, we define some intermediate quantities. As is shown above, there are three hub vertices in $I_g$ for $g\ge2$. According to the definition of the vertex cover, any vertex cover must contain at least two hub vertices. Then all vertex covers of the $I_g$ can be sorted into two classes: $D_g^1$, $D_g^2$, where $D_g^k$, $k=1,2,$ represent those vertex covers, each of which includes exactly $k+1$ hub vertices. Let $\gamma_g^k$ be the minimum cardinalitry among all the elements in the $D_g^k$. Let $\Theta_g^k$ be a element in the $D_g^k$ that has the minimum cardinality.

\begin{lemma}\label{leSGDom03}
The vertex cover number of  $\phi_g, g\ge 2,$ is $\phi_g=min\{\gamma_g^1, \gamma_g^2\}$.
\end{lemma}

After reducing the problem of determining $\phi_g$ to computing $\gamma_g^1, \gamma_g^2$, next we evaluate $\gamma_g^1$ and$ \gamma_g^2$ using the fact that the network is self-similar.

\begin{lemma}\label{leSGDom04}
 For two successive generation networks $I_g$ and $I_{g+1}$,$g\ge2$,
 \begin{small}
 \begin{align}\label{Dom01}
 &\gamma_{g+1}^1=min\{3\gamma_g^1-2,2\gamma_g^1+\gamma_g^2-3\},
\end{align}
\begin{align}\label{Dom02}
&\gamma_{g+1}^2=min\{\gamma_g^1+2\gamma_g^2-3, 3\gamma_g^2-3\},
\end{align}
\end{small}
\end{lemma}

\begin{figure}
	\begin{center}
		\includegraphics[width=0.8\linewidth]{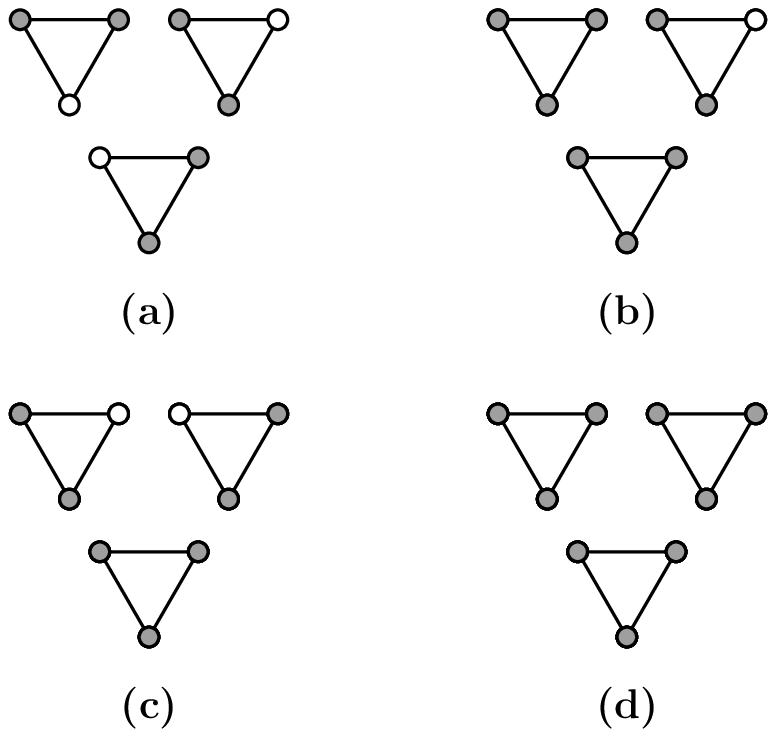}
	\end{center}
	\caption[kurzform]{Illustrations of all possible constructions of $\Theta_{g+1}^1$ and $\Theta_{g+1}^2$. Only the hub vertices of $I_g$ are shown. Open vertices are not in the dominating sets. Solid vertices are in the dominating sets} \label{F7}
\end{figure}

\begin{proof} We prove the lemma graphically, see Fig.~\ref{F7}. Note that $I_{g+1}$ is composed of three copies $I_g^\theta$, $\theta=1,2,3$ of $I_g$. By definition, two of the three hub vertices of $I_{g+1}$ are in $\Theta_{g+1}^1$. So those corresponding vertices of $I_g^\theta$ are in those corresponding vertex covers. Thus we can construct $\Theta_{g+1}^1$ from $\Theta_g^1, \Theta_g^2$ by considering whether the other three hub vertices of $I_g^\theta$ are in $\Theta_{g+1}^1$ or not, we obtain Eq.~\eqref{Dom01}.
For Eq.~\eqref{Dom02}, we can prove it similarly.
\end{proof}

\begin{lemma}\label{leSGDom05}
For network $I_g, g\ge 3, \gamma_g^1\ge \gamma_g^2.$
\end{lemma}
\begin{proof}
We prove this lemma by mathematical induction on g. For $g=3$, we can obtain $\gamma_3^1=4, \gamma_3^2=3$ by hand. Thus the basis step holds immediately. Assuming that the lemma holds for $g=t,t\ge3.$ Then, from Eq.~\eqref{Dom01}, we have $\gamma_{g+1}^1=min\{3\gamma_g^1-2, 2\gamma_g^1+\gamma_g^2-2\}.$ By induction hypothesis, we have
\begin{equation}
\gamma_{g+1}^1=2\gamma_g^1+\gamma_g^2-2.
\end{equation}
Analogously, we can obtain the following relations: 
\begin{equation}
\gamma_{g+1}^2=3\gamma_g^2-3.
\end{equation}
Using the induction hypothesis again, we have $\gamma_{g+1}^1\ge \gamma_{g+1}^2.$ Therefore, the lemma is true for g=t+1. This concludes the proof of the lemma.
\end{proof}

\begin{theorem}\label{thm1}
The k-power domination number of $\mathcal{G}_g,$ is
$$\gamma_{P,k}(\mathcal{G}_g)=\left\{\begin{aligned}
&\frac{3^{g-n-1}+3}{2}&, g\ge n+1\\
&1&, g\le n\end{aligned}\right.$$, where the integer n is a function of k:
\begin{equation}
n=[1+log_2(k+1)]
\end{equation}
\end{theorem}

\begin{proof} According to the Lemma\ref{leSGDom05}, $\gamma_g^1> \gamma_g^2, g\ge 3.$ Using the Eq.\ref{Dom01}, we obtain
\begin{equation}
\phi_{g+1}=\gamma_{g+1}^2=min\{2\gamma_g^1+\gamma_g^2-2, 3\gamma_g^2-3\}=3\gamma_g^2-3=3\phi_g-2.
\end{equation}
Since the $\phi_3=\gamma_3^2=6,$ we get
\begin{equation}
\phi_g=\frac{3^{g-1}+3}{2}, g\ge 3.
\end{equation}
For $g=1,2,$ we obtain $\phi_1=1$, $\phi_2=2$ and $\phi_3=3$ by hand. For $g\le n, \gamma_{P,k}(\mathcal{G}_g)=1,$ because $P_{\mathcal{G}_g,k}^\infty(\{A_g\})=\mathcal{G}_g$ by Lemma \ref{leSGDom01}. Using the Lemma \ref{leSGDom02} we yield
\begin{equation}
\gamma_{P,k}(\mathcal{G}_g)=\phi_{g-n+1}=\frac{3^{g-n-1}+3}{2},
\end{equation}
which completes the proof of the theorem.
\end{proof}

\section{$K$-power domination number of Sierpi\'nski graph}

Sierpi\'nski graph is an important graph in the graph theory. Although the domination number, the independence number and the number of maximum independent sets of Sierpi\'nski graph have been determined, the $k$-power domination number in the Sierpi\'nski graph is still unknown. Next we determine the $k$-power domination number in the Sierpi\'nski graph, and compare the result with that of the above-studied scale-free network, trying to reveal the effect of scale-free property on the k-power domination number.

\subsection{Construction of Sierpi\'nski graph}
The Sierpi\'nski graph can be constructed in an iterative way. The first generation of Sierpi\'nski graph, denoted as $\mathcal{S}_1$, is a triangle containing three vertices and edges. Suppose the $(g-1)$-generation of Sierpi\'nski graph is defined, the $g$-generation of Sierpi\'nski graph is obtained by making three replicas of $\mathcal{S}_{g-1}$, denoted as $\mathcal{S}_{g-1}^{\theta}$, $\theta=1,2,3$, and merging these replicas at their outmost vertices which are the three vertices in the Sierpi\'nski graph with degree two. We denote the outmost vertices used above as follows: let $A_{g}^{\theta}$, $B_{g}^{\theta}$, and $C_{g}^{\theta}$ be the outmost vertices of $\mathcal{S}_g^{\theta}$. In the iterative step, $A_g^1$, $B_g^2$, and $C_g^3$ become the outmost vertices $A_{g+1}$, $B_{g+1}$ and $C_{g+1}$ of $\mathcal{S}_{g+1}$, see Fig.~\ref{F8}.

\begin{figure}
	\begin{center}
		\includegraphics[width=0.8\linewidth]{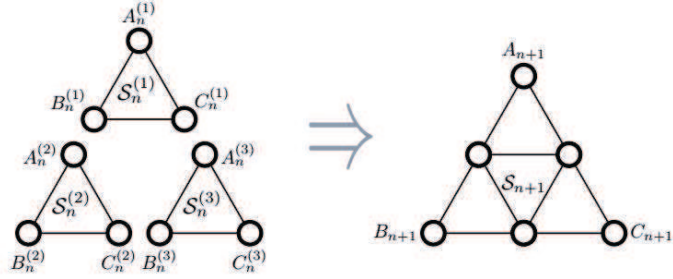}
	\end{center}
	\caption[kurzform]{The iteration process for the Sierpi\'nski graph.} \label{F8}
\end{figure}

By calculation, the vertices and edges of the Sierpi\'nski graph are exactly the same as the pseudofractal scale-free web, which are $N_g=(3^{g}+3)/2$ and $E_g=3^{g+1}$, respectively.

\subsection{$K$-power domination number}

Because degrees of vertices in $\mathcal{S}_g$ are 2 or 4, $\gamma_{P, k}(\mathcal{S}_g)=1$ hold for $k\ge4.$ Next we calculate $\gamma_{P, k}(\mathcal{S}_g)$ for $k=1,2,3.$

\begin{figure}
	\begin{center}
		\includegraphics[width=0.8\linewidth]{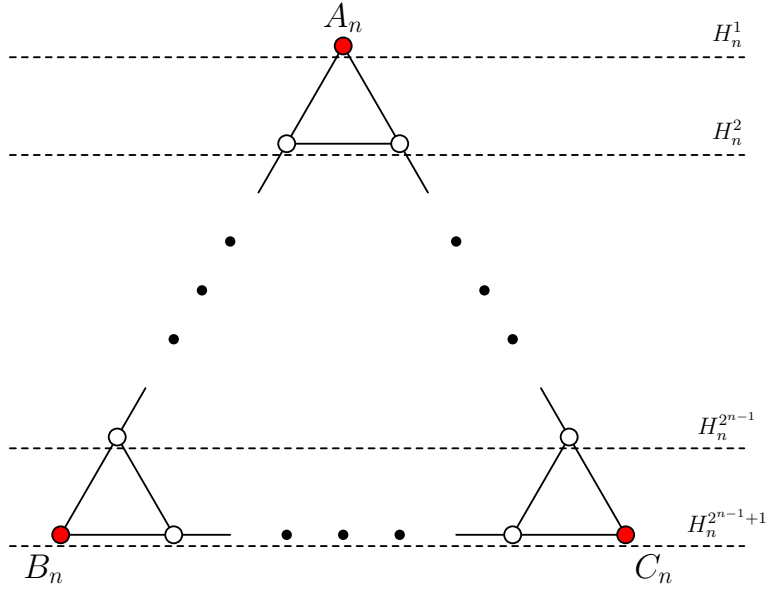}
	\end{center}
	\caption[kurzform]{Illustration of $H_g^i$, $1\le i\le 2^{g-1}+1$.} \label{F9}
\end{figure}

\begin{theorem}\label{thm02}
For $k=2,3, \gamma_{P, k}(\mathcal{S}_g)=1$
\end{theorem}

\begin{proof}
First we calculate $\gamma_{P, 2}(\mathcal{S}_g)$. If $\gamma_{P, 2}(\mathcal{S}_g)=1$, note that $\gamma_{P, 3}(\mathcal{S}_g)\le \gamma_{P, 2}(\mathcal{S}_g)$ and $\gamma_{P, 3}(\mathcal{S}_g)$ is a positive integer, we have that $\gamma_{P, 3}(\mathcal{S}_g)=1$. In the graph, we put the outmost vertex $A_g$ in the upmost position of the graph, see Fig.~\ref{F9}.
The vertices of $\mathcal{S}_g$ can be classified into $2^{g-1}+1$ classes: $H_g^i, 1\le i\le2^{g-1}+1,$ which consists of the vertices in the $i$th row of the graph following the direction from up to down. We denote $R_g^n$ as $\bigcup_{i=1}^nH_g^i.$ It is obvious that $P_{\mathcal{S}_g,2}^0(\{A_g\})=R_g^2.$ Next we prove that $P_{\mathcal{S}_g,2}^n(\{A_g\})=R_g^{n+2}$ holds for $0\le n\le2^{g-1}-1$ by induction on $n$. The basis step holds immediately, as is shown above. Suppose the assertion holds for $n=t$, according to the procedure of constructing $\mathcal{S}_g$, there are at most two vertices in $N(v)\setminus (V(\mathcal{S}_g)\setminus R_g^{t+2})$ for any $v$ in $H_g^{t+2}$. So we have 
\begin{equation}
P_{\mathcal{S}_g,2}^{t+1}(\{A_g\})=R_g^{t+3}.
\end{equation}
Therefore, $P_{\mathcal{S}_g,2}^n(\{A_g\})=R_g^{n+2}$ is true for $n=t+1.$ This concludes the proof of the lemma.
\end{proof}

Next, we will calculate $\gamma_{P, k}(\mathcal{S}_g)$ for $k=1$.

\begin{lemma}\label{leSGDom06}
The 1-power domination number for $\mathcal{S}_g$, $g\ge 2$, satisfies $\gamma_{P, 1}(\mathcal{S}_g)\ge \frac{3^{g-2}+1}{2}$.
\end{lemma}
\begin{proof}In fact, we can join $3^{g-2}$ copies, $\mathcal{S}^\theta_2, 1\le \theta\le3^{g-2}$, $g\ge 2$ of $\mathcal{S}_2$ at their outmost vertices to get $\mathcal{S}_g$ Since the degrees of the outmost vertices of $\mathcal{S}_2$ are 2 and S is a 1-power dominating set of $\mathcal{S}_g$, we must have $\mathcal{S}_2^\theta\cap S\neq\emptyset$. According to the construction of $\mathcal{S}_g$, any vertex of $\mathcal{S}_g$ can't be shared by more than two replicas of $\mathcal{S}_2$. Then we can evaluate the lower bound of $\gamma_{P, 1}(\mathcal{S}_g)$:
\begin{equation}
\gamma_{P, 1}(\mathcal{S}_g)\ge\frac{3^{g-2}+1}{2}.
\end{equation}
\end{proof}~\\

Next we define the condition 2 and condition 3. A subset of $V(\mathcal{S}_g)$ that satisfies the condition 3 is a 1-power domination set. A subset of $V(\mathcal{S}_g)$ that satisfies the condition 2 is not a 1-power domination set, but it helps to construct a 1-power domination set in $\mathcal{S}_{g+1}$.

\begin{definition}\label{deSGDom04}
A subset S of $V(\mathcal{S}_g)$ satisfies the condition 2 if it satisfies the following condition:
\par(1)The cardinality of S is $\frac{3^{g-2}+1}{2}$,
\par(2)If we remove $P_{\mathcal{S}_g, 1}^{\infty}(S)$ and the edges linked to these vertices, the graph left forms a path tree that starts from one outmost vertex and ends at another outmost vertex,
\par(3)The outmost vertex that is not in the path tree mentioned in (2) is in S.
\end{definition}

Fig.~\ref{F101112} shows a set satisfies the condition 2 and the path tree corresponds to it which is mentioned in the definition of condition 2. Note that if we see $\mathcal{S}_n$ as a component of $\mathcal{S}_g$ for $g> n$ and let S be a 1-power domination set of $\mathcal{S}_g$, then $S\cap V(\mathcal{S}_n)$ satisfies the conditions in the Definition~\ref{deSGDom04}. If $P_{\mathcal{S}_g, 1}^{\infty}(S)$ contains a vertex in the path tree (defined in the Definition~\ref{deSGDom04}), then all the vertices in the path tree will be in the $P_{\mathcal{S}_g, 1}^{\infty}(S)$ considering the degree of every vertex in the path tree.

\begin{definition}\label{deSGDom05}
A subset $S$ of $V(\mathcal{S}_g)$ satisfies the condition 3, if S satisfies the following conditions:
\par(1)The cardinality of S is $\frac{3^{g-2}+1}{2}$,
\par(2)S satisfies $P_{\mathcal{S}_g,1}^\infty(S)=\mathcal{S}_g$,
\par(3)Three outmost vertices of $\mathcal{S}_g$ are not in S.
\end{definition}

Fig.~\ref{F101112} shows a set that satisfies the condition 3.
\begin{lemma}\label{leSGDom07}
There exist subsets $\Theta_g^1, \Theta_g^2$ of $V(\mathcal{S}_g)$  for $g\ge3$. $\Theta_g^1$ Satisfies the condition 2 and $\Theta_g^2$ satisfies the condition 3.
\end{lemma}

\begin{figure}
	\begin{center}
		\includegraphics[width=0.8\linewidth]{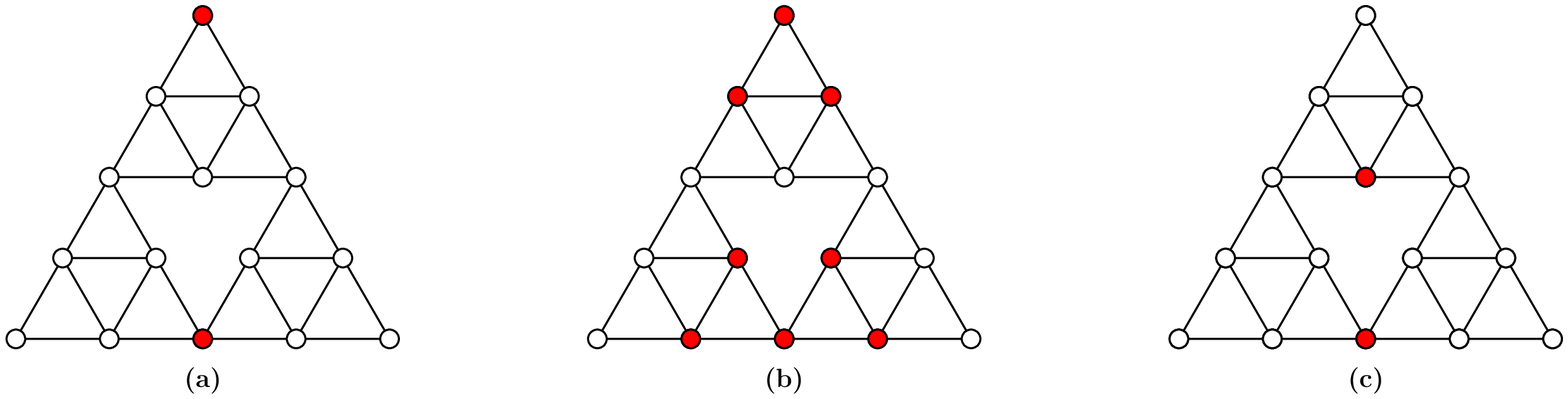}
	\end{center}
	\caption[kurzform]{In (a), the set consists of solid points is a set that satisfies the condition 2. In (b), the set consists of open points constructs a path tree that is mentioned in the definition of condition 2. In (c), the set consists of solid points is a set that satisfies the condition 3.} \label{F101112}
\end{figure}

\begin{proof}We prove the lemma by induction on g. The lemma holds immediately for $g=3$, see Fig.~\ref{F101112}, so the basis step holds immediately. Suppose the lemma holds for $g=t$. As is shown above, we can join three replicas of $\mathcal{S}_t$, $\mathcal{S}_t^\theta,$ $\theta=1,2,3,$ at their outmost vertices to get $\mathcal{S}_{t+1}$. In the process, we join $B_t^1$ and $A_t^2$ together and join $B_t^3$ and $C_t^2$ together and join $C_t^1$ and $A_t^3$ together. By induction, we can find a set $S^1$ in $\mathcal{S}_t^1$ that satisfies the condition 3. We can also find a set $S^2$ in $\mathcal{S}_t^2$ that satisfies the condition 2 whose path tree links $A_t^2$ and $B_t^2$. We can also find a set $S^3$ in $\mathcal{S}_t^3$, that satisfies the condition 2 whose path tree links $A_t^3$ and $C_t^3$. We merge $S^\theta, \theta=1,2,3,$ together to get $\Theta_{t+1}^2,$ a subset of $V(\mathcal{S}_{t+1})$.

Next we prove that $\Theta_{t+1}^2$ satisfies the condition 3. Because $C_t^2$(resp.$B_t^3$) is in the $S^2$(resp.$S^3$) and we merge $C_t^2$ and $B_t^3$ together, we have
\begin{equation}
|\Theta_{t+1}^2|=3(\frac{3^{t-2}+1}{2})-1=\frac{3^{t-1}+1}{2},
\end{equation}
so $\Theta_{t+1}^2$ satisfies (1). Note that the path tree has the property: if one vertex of the path tree is in the $P_{\mathcal{S}_{t+1},1}^\infty(S),$ then all the vertices of the path tree are in the $P_{\mathcal{S}_{t+1},1}^\infty(S).$ Because $B_t^1\in P_{\mathcal{S}_t^1,1}^\infty(\mathcal{S}^1)\subset P_{\mathcal{S}_{t+1},1}^\infty(\Theta_{t+1}^2)$, we have that all the vertices of the path tree of $\mathcal{S}_t^2$ are in the $P_{\mathcal{S}_{t+1},1}^\infty(\Theta_{t+1}^2)$ using the property of the path tree, so we get
\begin{equation}
V(\mathcal{S}_t^2)\subset P_{\mathcal{S}_{t+1},1}^\infty(\Theta_{t+1}^2).
\end{equation}
We can also get
\begin{equation}
V(\mathcal{S}_t^3)\subset P_{\mathcal{S}_{t+1},1}^\infty(\Theta_{t+1}^2)
\end{equation}
for the same reason. So we have that
\begin{equation}
V(\mathcal{S}_{t+1})=P_{\mathcal{S}_{t+1},1}^\infty(\Theta_{t+1}^2),
\end{equation}
which means that $\Theta_{t+1}^2$ satisfies (2). Since $A_t^1$ is not in the $V(\mathcal{S}_t^1)$, $B_t^2$ is not in the $V(\mathcal{S}_t^2)$ and $C_t^3$ is not in the $V(\mathcal{S}_t^3),$ then three outmost vertices of $\mathcal{S}_{t+1}$ are not in $\Theta_t^2.$ Then $\Theta_{t+1}^2$ satisfies (3). In a word, $\Theta_{t+1}^2$ satisfies the condition 3.

Next we construct a set that satisfies the condition 2. By induction, we can find $T^1\subset V(\mathcal{S}_t^1)$ whose path tree links $B_t^1$ and $C_t^1.$ We can find $T^2\subset V(\mathcal{S}_t^2)$ whose path tree links $A_t^2$ and $B_t^2.$ We can also find $T^3\subset V(\mathcal{S}_t^3)$ whose path tree links $A_t^3$ and $C_t^3.$ We merge $T^\theta, \theta=1,2,3,$ to get $\Theta_{t+1}^1.$ It is easy to verify that
\begin{equation}
|\Theta_{t+1}^1|=3(\frac{3^{t-2}+1}{2})-1=\frac{3^{t-1}+1}{2},
\end{equation}
which means $\Theta_{t+1}^1$ satisfies (1). When the vertices of $P_{\mathcal{S}_{t+1},1}^\infty(\Theta_{t+1}^1)$ and the edges linked to these vertices are removed from $\mathcal{S}_{t+1},$ the graph left forms a path tree linking $B_{t+1}$(which coincides with $C_t^3$) and $C_{t+1}$(which coincides with $C_t^3$) which means $\Theta_{t+1}^1$ satisfies (2). The fact that $\Theta_{t+1}^1$ satisfies (3) is obvious. So the lemma is true for $g=t+1.$ This concludes the proof of the lemma.
\end{proof}

\begin{theorem}\label{thm03}
The $k$-power domination number of $\mathcal{S}_g$ is
$${\gamma_{P, k}(\mathcal{S}_g)}=\left\{
\begin{aligned}
\frac{3^{g-2}+1}{2}&,& g\ge 2, k=1\\
1&,& g=1, k=1\\
1&,& k\ge2
\end{aligned}
\right.
$$.
\end{theorem}

\begin{proof} According to the Lemma \ref{leSGDom07}, for any $g\ge3,$ there exists a subset $\Theta_g^2$ of $V(\mathcal{S}_g)$, which satisfies the condition 3. According to the definition, $\Theta_g^2$ is a 1-power dominating set of $\mathcal{S}_g.$ So we have
\begin{equation}
\gamma_{P, 1}(\mathcal{S}_g)\le\frac{3^{g-1}+1}{2}, g\ge3.
\end{equation}
Using the Lemma \ref{leSGDom06}, we can get
\begin{equation}
\gamma_{P, 1}(\mathcal{S}_g)=\frac{3^{g-1}+1}{2}, g\ge3.
\end{equation}
We can obtain $\gamma_{P, 1}(\mathcal{S}_1)=1$ and $\gamma_{P, 1}(\mathcal{S}_2)=1$ by hand. We have shown above in the Theorem~\ref{thm02} that
\begin{equation}
\gamma_{P, k}(\mathcal{S}_g)=1, k\ge2.
\end{equation}
This concludes the proof of the theorem.
\end{proof}

\section{Comparison and analysis}
In this paper, we determine the k-power domination number of two graphs of the same number of vertices and edges: pseudofractal scale-free web and the Sierpi\'nski gasket. For $k=1$, the k-power domiantion number of the pseudofractal scale-free web and the Sierpi\'nski graph are similar, both growing at an exponential rate with respect to the number of vertices. However, for $k\ge 2$, the k-power domination number of the Sierpi\'nski graph of all generations is $1$. On the contrary, as the number of generation gets bigger, the k-power domination number of the pseudofractal scale-free web is $1$ at first, but grows at an exponential rate with respect to the number of vertices later on. We maintain that difference in the behavior of k-power domination numbers can be attributed to the structual distinction between the two graphs.\\

For $k=1$, since any vertex in the pseudofractal scale-free web has a degree strictly bigger than $1$, any given vertex in the graph can only 1-power dominate a small part of the graph. As a result, as the number of vertices increases, we need more vertices to 1-power dominate the whole graph. Moreover, the 1-power domination number of the $(g+1)$-generation of the graph is almost three times that of the $g$-generation of the graph since the $(g+1)$-generation of the graph can be constructed by merging three replicas of the $g$-generation of the graph. The 1-power domination number of the Sierpi\'nski graph can also be explained like above.\\

For $k\ge 2$, note that the degrees of vertices in all generations of the Sierpi\'nski graph are uniformly bounded, while the degree distribution of vertices of the pseudofractal scale-free web obeys a power law form $P(g)\sim g^{1+\frac{ln3}{ln2}}$. When there are more vertices with big degrees, it is more likely that the propagation in the definition of k-power domination set will end at these vertices since we can only include the neighbors of an vertex if few of its neighbors are not k-power dominated and this is hard to be achieved if it has a lot of neighbors. As a result, we need more vertices to k-power dominate the whole graph, which accounts for the phenomena that as the number of vertices grows, the k-power domination number of the pseudofractal scale-free web grows at an exponential rate in the end.\\

Although we only determine the k-power domination number in a scale-free network. It is natural to believe that k-power domination number of other scale-free graphs have behaviors similar to the pseudofractal scale-free web.

\section{Conclusions}
In conclusion, the difference in the k-power domincation number in these two graphs indicates the heterogeneity of the pseudofractal scale-free web and the homogeneity of the Sierpi\'nski graph.

\section*{References}

\bibliographystyle{plain}

\bibliography{k-power}

\begin{thebibliography}{10}

\bibitem{abe_scale-free_2004-1}
S~Abe and N~Suzuki.
\newblock Scale-free network of earthquakes.
\newblock {\em Europhysics Letters (EPL)}, 65(4):581--586, February 2004.

\bibitem{barrera_power_2011}
Roberto Barrera and Daniela Ferrero.
\newblock Power domination in cylinders, tori, and generalized {Petersen}
  graphs.
\newblock {\em Networks}, 58(1):43--49, August 2011.

\bibitem{chang_$$varveck$$_2015}
Gerard~Jennhwa Chang and Nicolas Roussel.
\newblock On the \$\${\textbackslash}varvec\{k\}\$\$ k -power domination of
  hypergraphs.
\newblock {\em Journal of Combinatorial Optimization}, 30(4):1095--1106,
  November 2015.

\bibitem{dorbec_power_2008}
Paul Dorbec, Michel Mollard, Sandi Klavžar, and Simon Špacapan.
\newblock Power {Domination} in {Product} {Graphs}.
\newblock {\em SIAM Journal on Discrete Mathematics}, 22(2):554--567, January
  2008.

\bibitem{dorfling_note_2006}
Michael Dorfling and Michael~A. Henning.
\newblock A note on power domination in grid graphs.
\newblock {\em Discrete Applied Mathematics}, 154(6):1023--1027, April 2006.

\bibitem{dorogovtsev_pseudofractal_2002}
S.~N. Dorogovtsev, A.~V. Goltsev, and J.~F. Mendes.
\newblock Pseudofractal scale-free web.
\newblock {\em Physical Review E Statistical Nonlinear \& Soft Matter Physics},
  65(6 Pt 2):066122, 2002.

\bibitem{guo_improved_2008}
Jiong Guo, Rolf Niedermeier, and Daniel Raible.
\newblock Improved {Algorithms} and {Complexity} {Results} for {Power}
  {Domination} in {Graphs}.
\newblock {\em Algorithmica}, 52(2):177--202, October 2008.

\bibitem{haynes_domination_2002}
Teresa~W. Haynes, Sandra~M. Hedetniemi, Stephen~T. Hedetniemi, and Michael~A.
  Henning.
\newblock Domination in {Graphs} {Applied} to {Electric} {Power} {Networks}.
\newblock {\em Siam Journal on Discrete Mathematics}, 15(4):519--529, 2002.

\bibitem{kim_weighted_2002}
Hyun-Joo Kim, Youngki Lee, Byungnam Kahng, and In-mook Kim.
\newblock Weighted {Scale}-{Free} {Network} in {Financial} {Correlations}.
\newblock {\em Journal of the Physical Society of Japan}, 71(9):2133--2136,
  September 2002.

\bibitem{liao_power_2013}
Chung~Shou Liao and D.~T. Lee.
\newblock Power {Domination} in {Circular}-{Arc} {Graphs}.
\newblock {\em Algorithmica}, 65(2):443--466, 2013.

\bibitem{wang_power_2005}
Chung-Shou Liao and Der-Tsai Lee.
\newblock Power {Domination} {Problem} in {Graphs}.
\newblock In Lusheng Wang, editor, {\em Computing and {Combinatorics}}, volume
  3595, pages 818--828. Springer Berlin Heidelberg, Berlin, Heidelberg, 2005.

\bibitem{rajan_2-power_2015}
R.~Sundara Rajan, J.~Anitha, and Indra Rajasingh.
\newblock 2-{Power} {Domination} in {Certain} {Interconnection} {Networks}.
\newblock {\em Procedia Computer Science}, 57:738--744, 2015.

\bibitem{zhang_exact_2009}
Z.~Zhang, Y.~Qi, S.~Zhou, W.~Xie, and J.~Guan.
\newblock Exact solution for mean first-passage time on a pseudofractal
  scale-free web.
\newblock {\em Physical Review E Statistical Nonlinear \& Soft Matter Physics},
  79(1):021127, 2009.

\end{thebibliography}

\end{document}